\documentclass[11pt]{article}

\usepackage[T1]{fontenc}
\usepackage[utf8]{inputenc}
\usepackage{lmodern}
\usepackage{microtype}
\usepackage{geometry}
\usepackage{amsmath,amssymb,amsthm,mathtools}
\usepackage{xcolor}
\usepackage{enumerate}
\usepackage{hyperref}
\hypersetup{
  colorlinks=true,
  linkcolor=blue!55!black,
  citecolor=green!35!black,
  urlcolor=blue!65!black,
  pdftitle={A Lovász Theta Parameter and Theta Body for Signed Graphs},
  pdfauthor={Gabriel Coutinho}
}

\newtheorem{theorem}{Theorem}[section]
\newtheorem{lemma}[theorem]{Lemma}
\newtheorem{proposition}[theorem]{Proposition}
\newtheorem{corollary}[theorem]{Corollary}

\newtheorem{problem}[theorem]{Problem}

\theoremstyle{definition}
\newtheorem{definition}[theorem]{Definition}
\newtheorem{example}[theorem]{Example}
\theoremstyle{remark}

\newcommand{\R}{\mathbb{R}}

\newcommand{\one}{\mathbf{1}}
\newcommand{\ip}[2]{\left\langle #1,#2\right\rangle}
\newcommand{\norm}[1]{\left\lVert #1\right\rVert}

\newcommand{\cB}{\mathcal{B}}
\newcommand{\cS}{\mathcal{S}}
\newcommand{\cH}{\mathcal{H}}
\newcommand{\conv}{\operatorname{conv}}
\newcommand{\diag}{\operatorname{diag}}
\newcommand{\Diag}{\operatorname{Diag}}
\newcommand{\Tr}{\operatorname{Tr}}
\newcommand{\tbar}{\overline{\vartheta}}
\newcommand{\tb}{\overline{\vartheta}_{\mathrm b}}
\newcommand{\chib}{\chi_{\mathrm b}}
\newcommand{\chifb}{\chi_{\mathrm{fb}}}
\newcommand{\betab}{\beta_{\mathrm b}}
\newcommand{\STAB}{\operatorname{STAB}}
\newcommand{\THbody}{\operatorname{TH}}
\newcommand{\QSTAB}{\operatorname{QSTAB}}
\newcommand{\BSTAB}{\operatorname{BSTAB}}
\newcommand{\BTH}{\operatorname{BTH}}
\newcommand{\BQSTAB}{\operatorname{BQSTAB}}
\newcommand{\BIP}{\operatorname{BIP}}
\newcommand{\cart}{\mathbin{\Box}}

\title{A Lovász Theta Parameter and Theta Body for Signed Graphs}
\author{Gabriel Coutinho\thanks{Department of Computer Science at the Federal University of Minas Gerais, Brazil. Email: gabriel@dcc.ufmg.br.}}
\date{\today}

\begin{document}

\maketitle

\begin{abstract}
We introduce a Lov\'asz-type semidefinite parameter for balanced colouring of signed graphs. Its homomorphism target is a unit sphere equipped with an orthogonal involution: the fixed and anti-fixed components play different roles, while applying the involution to one endpoint realizes switching. The resulting parameter admits a symmetric formulation with two positive semidefinite matrices and an equally symmetric dual. It is also one half of the strict vector chromatic number of the ordinary graph formed by the negative edges of the double switching graph. Our second main contribution is a hierarchy of convex corners. Starting from the balanced induced subgraph polytope, we define signed analogues of the stable-set, theta, and clique-inequalities relaxations in the original vertex space. The signed theta body has an intrinsic two-matrix description, its all-ones gauge is the new scalar parameter, and the ordinary stable-set hierarchy is recovered exactly from signed digon graphs. For all-negative signatures, the construction becomes a relaxation of the maximum induced bipartite subgraph problem and is related to the generalized theta number. Finally, we propose a notion of balanced perfectness and show that it is strictly weaker than perfectness of the associated double-cover graph.
\end{abstract}

\medskip
\noindent\textbf{Keywords.}
Signed graph; graph homomorphism; Lov\'asz theta function; semidefinite
programming; balanced colouring; balanced induced subgraph; convex corner.

\medskip
\noindent\textbf{2020 Mathematics Subject Classification.}
05C15, 05C22, 05C50, 90C22.


\section{Introduction and motivation}

The Lov\'asz theta function occupies a distinguished position between graph colouring, semidefinite optimization, and geometric representation of graphs. On the chromatic side, the number $\tbar(G):=\vartheta(\overline G)$ admits a homomorphic definition: it is the least $q \in \mathbb{R}$ so that vertices of $G$ can be represented by unit vectors in Euclidean space with adjacent vertices yielding a prescribed inner product $1/(1-q)$. This viewpoint is classical; see Lov\'asz~\cite{LovaszShannon}, Knuth~\cite{KnuthSandwich}, Karger--Motwani--Sudan~\cite{KargerMotwaniSudan}, and Godsil et al.~\cite{godsil2019graph}.

Our first purpose is to develop an analogous construction for \emph{balanced colouring} of a signed graph. A balanced colour class need not be stable: it may contain many edges, provided that the induced signed graph has no negative cycle. Balanced colouring has recently been studied through signed homomorphisms, Hadwiger-type questions, fractional colourings, and signed analogues of Kneser, Schrijver, and Borsuk graphs; see \cite{JimenezEtAl,HuEtAlFractional,KuffnerEtAl,KuffnerEtAlFractionalColoring}. The polyhedral optimization problem underlying one colour class is older: the convex hull of balanced induced vertex sets was introduced and studied by Barahona and Mahjoub~\cite{BarahonaMahjoub1989,BarahonaMahjoub1994}.

A signed spherical target must satisfy two compatibility requirements. First, switching at a vertex reverses the sign of every incident edge. Consequently, the target should carry an operation that, when applied to the image of one endpoint, interchanges the positive- and negative-edge relations. An orthogonal involution $R$ provides precisely such an operation. Indeed, we shall propose that negative and positive adjacency between $u$ and $v$ are encoded respectively by
\[
  \langle x_u,x_v\rangle=\gamma
  \qquad\text{and}\qquad
  \langle x_u,R x_v\rangle=\gamma,
\]
then replacing $x$ by $Rx$ exchanges the two conditions, because
\[
  \langle Rx,y\rangle=\langle x,Ry\rangle
  \quad\text{and}\quad
  \langle Rx,Ry\rangle=\langle x,y\rangle.
\]
Thus the action $x\mapsto Rx$ geometrically realizes switching at a vertex.

Second, the homomorphism formulation of balanced colouring naturally uses targets with a positive loop at every vertex. In the spherical model this is compatible to every point $x$ of the target satisfying the positive adjacency condition to itself, that is,
\[
  \langle x,Rx\rangle=\gamma.
\]
The ordinary antipodal map $R=-I$ is too restrictive, since then $\langle x,Rx\rangle=-1$ for every unit vector $x$, hence permitting only $\gamma=-1$. To obtain a nontrivial family for $\gamma_t = \frac{1}{(1-2t)}$ and $t > 1$, we instead allow $R$ to have both fixed and anti-fixed subspaces, and this is the key step in our construction. 

Writing $x=x_++x_-$ according to this decomposition gives
\[
  \langle x,Rx\rangle=\|x_+\|^2-\|x_-\|^2,
\]
so the required value $\gamma_t$ can be realized by restricting the unit sphere to a suitable latitude.

\begin{quote} \itshape
  Given a signed graph $\Sigma$, we will define $\tb(\Sigma)$ to be the least $t\geq1$ for which there is a signed homomorphism from $\Sigma$ to the graph $\cS_t(\cH,R)$, where $\cH$ is a finite-dimensional Euclidean space, $R$ is an orthogonal involution, the vertex set of $\cS_t(\cH,R)$ is
  \[ \{x \in \cH : \|x\| = 1, \langle x, R x \rangle = \gamma_t\}, \]
  and the adjacency relations are defined by
  \begin{align*}
  x\sim_-y
  &\quad\Longleftrightarrow\quad
  \ip{x}{y}=\gamma_t,\\
  x\sim_+y
  &\quad\Longleftrightarrow\quad
  \ip{x}{Ry}=\gamma_t.
\end{align*}
\end{quote}

The parameter $\tb(\Sigma)$ has a particularly clean semidefinite formulation. It is the minimum $t$ for which there are two positive semidefinite matrices $X,Y$, both of order $|V(\Sigma)|$, with constant diagonals $t-1$ and $t$, and with the signed edge constraints 
\[
  X_{uv}-\sigma(uv)Y_{uv}=-1.
\]
We also derive its standard dual in two symmetric matrix variables. These programs are directly obtained from the two components of the spherical representation.

While our definition is obtained exclusively from directly examining the signed graph, we also show that it admits a purely ordinary graph interpretation, similarly to other parameters for signed graphs. Let $C(\Sigma)$ be the ordinary graph formed by the negative edges of the double switching graph, after positive loops are adjoined. Its vertices are $V(\Sigma) \times \{+,-\}$, which may be referred to as local switching states at $v$, that is $(v,+)$ and $(v,-)$, and two states are adjacent precisely when they cannot occur together in one balanced colour class. We prove
\[
  \tb(\Sigma)=\frac12\tbar(C(\Sigma)).
\]

Our second main point is the introduction of the convex body $\BTH(\Sigma)$, which plays the role of the theta body for balanced induced subgraphs. We show that it is a natural projection of $\THbody(C(\Sigma))$ onto the original vertex space, and that it is contained between the balanced induced subgraph polytope $\BSTAB(\Sigma)$ and the projected clique relaxation $\BQSTAB(\Sigma)$, which we also define. In other words, we show that the following signed analogue of the classical sandwich theorem holds:
\[
  \BSTAB(\Sigma)\subseteq\BTH(\Sigma)\subseteq\BQSTAB(\Sigma).
\]
The body $\BTH(\Sigma)$ has an intrinsic two-Gram semidefinite description; its all-ones gauge is exactly $\tb(\Sigma)$. This gives a body-level role to the reflection decomposition which is not visible from the scalar identity shown above. 

An immediate check that our proposal is a natural extension is that replacing every edge of an ordinary graph by a positive--negative digon recovers, exactly, the usual bodies $\STAB$, $\THbody$, and $\QSTAB$. The all-negative specialization provides another useful test. In this case, balanced induced subgraphs are precisely induced bipartite subgraphs, and
\[
  C(-G)=K_2\cart G.
\]
Consequently, maximizing over $\BTH(-G)$ gives $\vartheta(K_2\cart G)$, a natural semidefinite bound for the maximum $2$-colourable induced subgraph problem. This places the signed construction next to the generalized theta number $\vartheta_k(G)$ and the recent SDP literature on maximum $k$-colourable induced subgraphs \cite{KuryatnikovaSotirovVera,SinjorgoSotirov,BarkelSotirov}.

Finally, we use the projected bodies to propose a notion of \textit{balanced perfectness}. Perfectness of $C(\Sigma)$ implies balanced perfectness, but the converse fails on a four-vertex signed graph. In that example the signed theta body is already the exact balanced induced subgraph polytope and hence is polyhedral, while $C(\Sigma)$ contains an induced $5$-cycle. This gives a concrete sense in which the body theory developed in this paper is not merely perfect graph theory on $C(\Sigma)$ under a different notation.

The paper is organized as follows. Section~\ref{sec:prelim} fixes our conventions and recalls necessary results. Section~\ref{sec:sphere} introduces the signed sphere and its two-vector form. Section~\ref{sec:sdp} gives the primal and dual semidefinite programs. Sections~\ref{sec:cover} and~\ref{sec:equivalence} study the conflict graph and prove the theta identity. Section~\ref{sec:corners} develops the signed stable-set hierarchy. Section~\ref{sec:perfect} treat the perfectness specializations.


\section{Preliminaries}
\label{sec:prelim}

We refer to Zaslavsky~\cite{ZaslavskySignedGraphs} for the basic theory of signed graphs, including balance and switching. Our conventions concerning homomorphisms follow Naserasr, Rollov\'a, and Sopena~\cite{NaserasrRollovaSopena}; see also~\cite{NaserasrSopenaZaslavsky}.

\subsection{Signed graphs, switching, and balance}

A \emph{signed graph} is a pair $\Sigma=(G,\sigma)$, where $G$ is a finite graph and $\sigma:E(G)\to\{+1,-1\}$. We allow positive loops and exclude negative loops. Parallel edges of opposite signs are allowed and will be called a \emph{digon}. Repeated edges of the same sign are immaterial. The sign of a closed walk is the product of the signs of its edges, counted with multiplicity.

A \emph{switching function} is a map $\tau:V(G)\to\{+1,-1\}$. The switched signature is
\[
  \sigma^\tau(uv)=\tau(u)\sigma(uv)\tau(v).
\]
For a loop at $v$, this formula leaves its sign unchanged. Two signatures are \emph{switching equivalent} if one is obtained from the other in this way.

A signed graph is \emph{balanced} if every cycle has positive sign. A set $B\subseteq V(\Sigma)$ is balanced if $\Sigma[B]$ is balanced. We shall use the following classical characterization.

\begin{theorem}[Harary~\cite{HararyBalance}] \label{thm:harary}
A signed graph $\Sigma=(G,\sigma)$ is balanced if and only if there is a map $h:V(G)\to\{+1,-1\}$ such that
\[
  \sigma(uv)=h(u)h(v)
\]
for every edge $uv$. Equivalently, $\Sigma$ can be switched to the all-positive signature.
\end{theorem}

We use a convention tailored to balanced colouring, as in \cite{KuffnerEtAl}. For a signed graph $\Sigma$, let $\Sigma^{+\ell}$ be obtained by adding a positive loop at every vertex. Positive loops do not affect balance, balanced colourings, or switching equivalence. 

\begin{quote}\itshape
Unless stated otherwise, every signed graph is replaced by its positive-loop completion.
\end{quote}
The notation $\Sigma^{+\ell}$ is used when this completion must be displayed explicitly.

\subsection{Signed homomorphisms and balanced colourings}

Let $\Sigma=(G,\sigma)$ and $\Pi=(H,\pi)$ be signed graphs. A \emph{signed homomorphism} $f:\Sigma\to\Pi$ consists of a switching function $\tau:V(G)\to\{\pm1\}$ and a vertex map $f:V(G)\to V(H)$ such that for every edge $e=uv$ of $\Sigma$, the graph $\Pi$ contains an edge between $f(u)$ and $f(v)$ of sign
\[
  \sigma^\tau(e)=\tau(u)\sigma(e)\tau(v).
\]
This is equivalent to preservation of the sign of every closed walk; see \cite{BrewsterFoucaudHellNaserasr}.

A \emph{balanced $p$-colouring} of $\Sigma$ is a partition of $V(\Sigma)$ into at most $p$ balanced sets. The least such $p$ is the \emph{balanced chromatic number} $\chib(\Sigma)$.

For $p\geq1$, define the \emph{balanced complete signed graph} $\cB_p$ as follows: its vertex set is $[p]$, every vertex has a positive loop, and every pair of distinct vertices is joined by one positive and one negative edge. The next proposition is observed in~\cite{JimenezEtAl}.

\begin{proposition}
\label{prop:balanced-hom}
A signed graph $\Sigma$ has a balanced $p$-colouring if and only if there is a signed homomorphism $\Sigma\to\cB_p$. Consequently,
\[
  \chib(\Sigma)=\min\{p:\Sigma\longrightarrow\cB_p\}.
\]
\end{proposition}

\begin{proof}
Suppose $V(\Sigma)=B_1\cup\cdots\cup B_p$, with each $B_i$ balanced. For each nonempty $B_i$, choose a switching that makes every edge of $\Sigma[B_i]$ positive, and combine these switchings into a global function. Map every vertex of $B_i$ to $i$. A monochromatic edge maps to the positive loop at $i$, while an edge between two colour classes maps to an edge of the required sign because $\cB_p$ contains both signs.

Conversely, suppose that, after switching, a map $f:\Sigma\to\cB_p$ preserves signs. Every edge with both ends in one fibre $f^{-1}(i)$ must be positive, because $\cB_p$ has no negative loop. Hence every fibre is balanced, and the fibres form a balanced $p$-colouring.
\end{proof}

A \emph{fractional balanced colouring} is a fractional cover by balanced vertex sets. If $\mathcal B(\Sigma)$ denotes the family of balanced subsets, then $\chifb(\Sigma)$ is the optimum of
\begin{align*}
  \text{minimize}\quad &\sum_{B\in\mathcal B(\Sigma)}y_B\\
  \text{subject to}\quad
  &\sum_{B\ni v}y_B\geq1 \qquad(v\in V(\Sigma)),\\
  &y_B\geq0 \qquad(B\in\mathcal B(\Sigma)).
\end{align*}
See~\cite{HuEtAlFractional,KuffnerEtAlFractionalColoring}.

\subsection{The ordinary Lov\'asz theta function and theta body}
\label{subsec:Lovasz-theta}

For a finite simple graph $G$ containing an edge, the strict vector chromatic number is the least $q\geq2$ for which there are unit vectors $x_v$ satisfying
\[
  \ip{x_u}{x_v}=\frac{1}{1-q}
  \qquad(uv\in E(G)).
\]
It coincides with $\vartheta(\overline G)$; see \cite{KargerMotwaniSudan}. We write
\[
  \tbar(G):=\vartheta(\overline G),
\]
and set $\tbar(G)=1$ when $G$ is edgeless.

Let $\mathsf S_q$ be the infinite graph whose vertices are unit vectors, with adjacency determined by the inner product $1/(1-q)$. The vector definition may be restated as follows; see~\cite{godsil2019graph}.

\begin{proposition}
\label{thm:strict-vector}
For every graph $G$ containing an edge,
\[
  \tbar(G)=\inf\{q\geq2:G\longrightarrow\mathsf S_q\}.
\]
\end{proposition}

In particular, $\tbar$ is monotone under graph homomorphisms. We shall also use the sandwich theorem
\[
  \omega(G)\leq\tbar(G)\leq\chi_f(G)\leq\chi(G);
\]
see~\cite{KnuthSandwich}.

The following primal--dual pair fixes our semidefinite programming (SDP) conventions. Here $J$ denotes the all-ones matrix, $\mathbb S^V$ is the space of real symmetric matrices indexed by $V$, and the inner product in $\mathbb S^V$ is given by $\langle A,B\rangle=\Tr(AB)$.
\begin{align}
  \tbar(G)
  = \qquad\qquad \min\quad &\lambda \notag\\
  \text{subject to}\quad
  &\lambda \in \mathbb{R}, M \in \mathbb S^V \notag, \\
  &M\succeq0,\notag\\
  &M_{vv}=\lambda-1 &&(v\in V(G)),\notag\\
  &M_{uv}=-1 &&(uv\in E(G)),
  \label{eq:tbar-primal}
\end{align}
and
\begin{align}
  \tbar(G)
  =\qquad\qquad \max\quad &\langle J,Z\rangle \notag\\
  \text{subject to}\quad
  & Z \in \mathbb S^V, \notag \\
  &Z\succeq0, \notag \\
  &\Tr Z=1,\notag\\
  &Z_{uv}=0 &&(u\neq v,\ uv\notin E(G)).
  \label{eq:tbar-dual}
\end{align}

A \emph{convex corner} in $\R^V$ is a compact convex subset of $\R^V_+$ with nonempty interior which is lower-comprehensive: if $x$ belongs to the corner and $0\leq y\leq x$, then $y$ also belongs to it. We recall the three convex corners associated with an ordinary graph. For a graph $G$, let
\[
  \STAB(G)=\conv\{\one_I:I\subseteq V(G)\text{ is stable}\},
\]
\[
  \QSTAB(G)=\{x\geq0:x(K)\leq1\text{ for every clique }K\text{ of }G\},
\]
and
\begin{equation}
\THbody(G)=\left\{x:\exists X\text{ such that }
\begin{pmatrix}1&x^{\mathsf T}\\x&X\end{pmatrix}\succeq0,
\ \diag X=x,
\ X_{uv}=0\ (uv\in E(G))\right\}.
\label{eq:theta-body}
\end{equation}
Then
\[
  \STAB(G)\subseteq\THbody(G)\subseteq\QSTAB(G).
\]
Moreover, if $C$ is a convex corner of $\R^V$ and $z\in\R^V$, the Minkowski functional (or gauge) of $C$ is $\gamma_C(z):=\inf\{\lambda>0:z\in\lambda C\}$. Then
\begin{equation}
  \tbar(G)=\gamma_{\THbody(G)}(\one).
\label{eq:gauge-theta}
\end{equation}
These facts are standard; see~\cite{GLS,KnuthSandwich}.


\section{Homomorphic definition of the balanced theta parameter}
\label{sec:sphere}

Let $\cH$ be a finite-dimensional real inner-product space and assume
\[
  \cH=\cH_+\oplus\cH_-.
\]
Define $R:\cH\to\cH$ to be the orthogonal involution
\[
  R(a,b)=(a,-b).
\]
For $t\geq1$, set
\[
  \gamma_t=\frac{1}{1-2t}.
\]
We define the signed spherical graph $\cS_t(\cH,R)$ on the vertex set
\begin{equation}
V(\cS_t(\cH,R))=
\{x\in\cH:\norm{x}=1,\ \ip{x}{Rx}=\gamma_t\}
\label{eq:vertex-set}
\end{equation}
by declaring
\begin{align*}
  x\sim_-y
  &\quad\Longleftrightarrow\quad
  \ip{x}{y}=\gamma_t,\\
  x\sim_+y
  &\quad\Longleftrightarrow\quad
  \ip{x}{Ry}=\gamma_t.
\end{align*}
Both relations may hold for a pair of distinct vertices, in which case the
target has a digon.

\begin{lemma}
\label{lem:R-switches}
For adjacent $x,y\in\cS_t(\cH,R)$, applying $R$ to exactly one endpoint interchanges the sign of the edge, while applying $R$ to both endpoints preserves it. Every vertex has a positive loop and no vertex has a negative loop.
\end{lemma}

\begin{proof}
Because $R$ is an orthogonal involution, it is self-adjoint. Hence
\[
  \ip{Rx}{y}=\ip{x}{Ry},
  \qquad
  \ip{Rx}{Ry}=\ip{x}{y}.
\]
The two adjacency relations are therefore interchanged by applying $R$ at one endpoint and unchanged by applying it at both. The latitude condition is exactly $x\sim_+x$. A negative loop would require $1=\gamma_t$, which is impossible for $t\geq1$.
\end{proof}

\begin{definition}
\label{def:theta-b}
For a nonempty signed graph $\Sigma$, with a positive loop adjoined at every vertex, define the \emph{balanced Lov\'asz theta parameter} by
\[
  \tb(\Sigma)
  :=\inf\left\{t\geq1:
  \Sigma\longrightarrow\cS_t(\cH,R)
  \text{ for some finite-dimensional }\cH\right\}.
\]
The arrow denotes a signed homomorphism.
\end{definition}

The bar indicates that this is a chromatic-side parameter, analogous to $\tbar(G)$. A switching in the source can be absorbed into the target.

\begin{proposition}
\label{prop:absorb-switching}
A signed homomorphism from $\Sigma=(G,\sigma)$ to $\cS_t(\cH,R)$ exists if and only if there are vectors $x_v\in\cS_t(\cH,R)$ such that, for every signed edge $uv$,
\begin{align*}
  \sigma(uv)=-1&\quad\Longrightarrow\quad
  \ip{x_u}{x_v}=\gamma_t,\\
  \sigma(uv)=+1&\quad\Longrightarrow\quad
  \ip{x_u}{Rx_v}=\gamma_t.
\end{align*}
\end{proposition}

\begin{proof}
One direction is immediate. Conversely, suppose that after switching the source by $\tau$, vectors $y_v$ give an edge-sign-preserving map. Set
\[
  x_v=
  \begin{cases}
    y_v,&\tau(v)=+1,\\
    Ry_v,&\tau(v)=-1.
  \end{cases}
\]
By Lemma~\ref{lem:R-switches}, applying $R$ at one endpoint toggles the edge sign and applying it at both preserves it. Thus the resulting vectors realize the original signature.
\end{proof}

Let
\[
  \cS_p=\cS_p(\R^{p-1}\oplus\R^p,R),
\]
where $R$ changes the sign of the second summand.

\begin{proposition}
\label{prop:Bp-embedding}
For every integer $p\geq1$, there is a signed homomorphism
\[
  \cB_p\longrightarrow\cS_p.
\]
Consequently,
\[
  \tb(\Sigma)\leq\chib(\Sigma)
\]
for every signed graph $\Sigma$.
\end{proposition}

\begin{proof}
The case $p=1$ is immediate. Assume $p\geq2$. Let $a_1,\ldots,a_p$ be the vertices of a regular simplex in $\R^{p-1}$, so that
\[
  \ip{a_i}{a_j}=\frac{1}{1-p}
  \quad(i\neq j),
\]
and let $b_1,\ldots,b_p$ be an orthonormal basis of $\R^p$. Define
\[
  x_i=
  \sqrt{\frac{p-1}{2p-1}}\,a_i
  \oplus
  \sqrt{\frac{p}{2p-1}}\,b_i.
\]
Then $\norm{x_i}=1$, $\ip{x_i}{Rx_i}=\gamma_p$, and, for $i\neq j$,
\[
  \ip{x_i}{x_j}=\ip{x_i}{Rx_j}=\gamma_p.
\]
Thus these points induce a copy of $\cB_p$ in $\cS_p$. The final assertion follows from Proposition~\ref{prop:balanced-hom}.
\end{proof}

The two coordinates of the target may always be conveniently separated. Put
\[
  \alpha_t=\sqrt{\frac{t-1}{2t-1}},
  \qquad
  \beta_t=\sqrt{\frac{t}{2t-1}}.
\]

\begin{lemma}
\label{lem:fixed-norms}
Write $x=x_+\oplus x_-\in\cH_+\oplus\cH_-$. Then
$x\in\cS_t(\cH,R)$ if and only if
\[
  \norm{x_+}^2=\frac{t-1}{2t-1} = \alpha_t^2,
  \qquad
  \norm{x_-}^2=\frac{t}{2t-1} = \beta_t^2.
\]
For $t>1$, every such point can be written $x=\alpha_t a\oplus\beta_t b$ with $a,b$ unit vectors. At $t=1$, the positive component vanishes.
\end{lemma}

\begin{proof}
The two defining equations are
\[
  \norm{x_+}^2+\norm{x_-}^2=1,
  \qquad
  \norm{x_+}^2-\norm{x_-}^2=-\frac{1}{2t-1}.
\]
Adding and subtracting gives the result.
\end{proof}

\begin{theorem}
\label{thm:two-vector}
For $t\geq1$, a signed graph $\Sigma=(G,\sigma)$ maps to $\cS_t(\cH,R)$ if and only if there are two families of unit vectors $(a_v)_{v\in V(G)}$ and $(b_v)_{v\in V(G)}$, in possibly different real inner-product spaces, such that
\begin{equation}
  (t-1)\ip{a_u}{a_v}
  -t\,\sigma(uv)\ip{b_u}{b_v}
  =-1
  \qquad(uv\in E(G)).
\label{eq:SV}
\end{equation}
\end{theorem}

\begin{proof}
By Proposition~\ref{prop:absorb-switching} and Lemma~\ref{lem:fixed-norms}, write $x_v=\alpha_t a_v\oplus\beta_t b_v$. On a negative edge the defining inner product is $\ip{x_u}{x_v}$, while on a positive edge it is $\ip{x_u}{Rx_v}$. Multiplication by $2t-1$ gives \eqref{eq:SV} in both cases. The converse follows by reversing the calculation. At $t=1$, the $a$-term vanishes.
\end{proof}


\section{Semidefinite programming formulation}
\label{sec:sdp}

The coefficients in Theorem~\ref{thm:two-vector} can be absorbed into the two Gram matrices. This gives the formulation we shall use throughout the paper.

\begin{theorem}
\label{thm:signed-theta-sdp}
For every nonempty signed graph $\Sigma$, the value $\tb(\Sigma)$ is the optimum of
\begin{equation*}
\begin{aligned}
  \text{minimize}\quad &t\\
  \text{subject to}\quad
  &X,Y \in \mathbb S^V, \qquad t \in \mathbb{R},\\ 
  &X\succeq0,\qquad Y\succeq0,\\
  &X_{vv}=t-1
    &&(v\in V(\Sigma)),\\
  &Y_{vv}=t
    &&(v\in V(\Sigma)),\\
  &X_{uv}-\sigma(uv)Y_{uv}=-1
    &&(uv\in E(\Sigma)).
\end{aligned}
\tag{$\mathrm P_\Sigma$}
\end{equation*}
\end{theorem}

\begin{proof}
Let $P,Q$ be the Gram matrices of the two unit-vector families in Theorem~\ref{thm:two-vector}. Setting
\[
  X=(t-1)P,
  \qquad
  Y=tQ
\]
gives a feasible solution of~$\mathrm P_\Sigma$. Conversely, a feasible solution has $t\geq1$ because $X\succeq0$ and $X_{vv}=t-1$. If $t>1$, the matrices $X/(t-1)$ and $Y/t$ are Gram matrices of unit vectors satisfying~\eqref{eq:SV}. If $t=1$, then $X=0$, while $Y$ itself is a Gram matrix.
\end{proof}

The standard dual keeps the same two-block symmetry. If both signs occur between $u$ and $v$, no off-diagonal relation is imposed in the dual.

\begin{theorem}
\label{thm:signed-theta-dual}
The dual of $\mathrm P_\Sigma$ is
\begin{equation*}
\begin{aligned}
  \text{maximize}\quad &\langle J,S\rangle\\
  \text{subject to}\quad
  &S,T \in \mathbb S^V,\\ 
  &S\succeq0,\qquad T\succeq0,\\
  &\Tr(S+T)=1,\\
  &S_{uv}=T_{uv}=0
    &&\bigl(u\neq v,\ uv\notin E_+(\Sigma)\cup E_-(\Sigma)\bigr),\\
  &S_{uv}+T_{uv}=0
    &&\bigl(uv\in E_+(\Sigma)\setminus E_-(\Sigma)\bigr),\\
  &S_{uv}-T_{uv}=0
    &&\bigl(uv\in E_-(\Sigma)\setminus E_+(\Sigma)\bigr).
\end{aligned}
\tag{$\mathrm D_\Sigma$}
\end{equation*}
The primal and dual optima are attained and equal.
\end{theorem}

\begin{proof}
For $u\neq v$, put
\[
  F^{uv}=\frac12(e_u e_v^{\mathsf T}+e_v e_u^{\mathsf T}).
\]
Assign unrestricted multipliers $\alpha_v,\beta_v$ to the two diagonal constraints and $z_e$ to the constraint indexed by an edge $e=uv$. The Lagrange dual requires
\[
  S=\Diag(\alpha)+\sum_{e=uv}z_e F^{uv}\succeq0,
  \qquad
  T=\Diag(\beta)-\sum_{e=uv}\sigma(e)z_e F^{uv}\succeq0,
\]
\[
  \sum_v(\alpha_v+\beta_v)=1,
\]
and has objective $\sum_v\alpha_v+\sum_e z_e$. Eliminating the multipliers gives exactly the support conditions in $\mathrm D_\Sigma$; moreover, $\sum_v\alpha_v+\sum_e z_e=\langle J,S\rangle$.

The dual is strictly feasible with $S=T=I/(2|V(\Sigma)|)$. The primal is feasible by Proposition~\ref{prop:Bp-embedding}, and a bounded level set is compact because all entries of positive semidefinite matrices are controlled by their diagonals. Standard SDP duality therefore gives equality and attainment.
\end{proof}


\section{The conflict graph}
\label{sec:cover}

The double switching graph $\operatorname{DSG}(\Sigma)$ is a signed graph on $V(\Sigma)\times\{\pm1\}$; see \cite{BrewsterGraves,NaserasrSopenaZaslavsky}. We use the ordinary spanning graph formed by its negative edges:
\[
  C(\Sigma):=\operatorname{DSG}(\Sigma)^{-}.
\]
This graph appears explicitly in the recent fractional balanced colouring literature~\cite{KuffnerEtAlFractionalColoring}. We call it the \emph{conflict graph} of $\Sigma$, emphasizing its optimization interpretation.

Its vertex set is $V(\Sigma)\times\{\pm1\}$. We write $\pi(v,s)=v$ for the natural projection onto $V(\Sigma)$. The adjacency rule is \begin{equation}
  (u,s)\sim(v,r)
  \quad\Longleftrightarrow\quad
  \text{there is an edge }e=uv \text{ in } \Sigma\text{ with }r=-\sigma(e)s.
\label{eq:conflict-adjacency}
\end{equation}
A positive edge $uv$ produces edges $(u,\pm) \sim (v,\mp)$, a negative edge $uv$ gives $(u,\pm) \sim (v,\pm)$, and the positive loop at $v$ produces the vertical edge $(v,+)(v,-)$. Thus two states are adjacent precisely when they cannot occur together in one balanced set with the displayed switching labels.

The fibre involution
\begin{equation}
  \iota(v,s)=(v,-s) \label{involution}
\end{equation}
is a fixed-point-free automorphism. Switching $\Sigma$ at a function $\tau$ induces the isomorphism $(v,s)\mapsto(v,\tau(v)s)$. A signed homomorphism $\Sigma\to\Pi$ also induces an ordinary homomorphism $C(\Sigma)\to C(\Pi)$ commuting with the fibre involutions.

\begin{proposition}
\label{prop:balanced-independent}
Independent sets of $C(\Sigma)$ are in correspondence with pairs $(B,h)$, where $B$ is a balanced set of $\Sigma$ and $h:B\to\{\pm1\}$ is a Harary labeling. In particular,
\[
  \alpha(C(\Sigma))=\betab(\Sigma),
\]
where $\betab(\Sigma)$ is the maximum cardinality of a balanced vertex set.
\end{proposition}

\begin{proof}
An independent set contains at most one of $(v,+)$ and $(v,-)$. Write it as $I=\{(v,h(v)):v\in B\}$. If $uv$ is an edge of $\Sigma[B]$, independence and~\eqref{eq:conflict-adjacency} give $h(v)\neq-\sigma(uv)h(u)$, hence
\[
  \sigma(uv)=h(u)h(v).
\]
Thus $h$ is a Harary labeling and $B$ is balanced. The converse follows by reversing the same argument.
\end{proof}

By a simple signed graph we mean one in which, between distinct vertices, there is at most one edge. The prescribed positive loops are disregarded when using the word simple.

The clique structure has an equally natural signed interpretation. A simple signed complete graph is \emph{antibalanced} if it can be switched to the all-negative edge signature. Let $\omega_{\mathrm{ab}}(\Sigma)$ denote the largest order of an antibalanced clique in a simple signed graph $\Sigma$.

\begin{proposition}
\label{prop:conflict-cliques}
If $\Sigma$ is a simple signed graph, then
\[
  \omega(C(\Sigma))=\max\{2,\omega_{\mathrm{ab}}(\Sigma)\}.
\]
More precisely, every clique of $C(\Sigma)$ with at least three vertices projects injectively onto an antibalanced clique of $\Sigma$, and every antibalanced clique has a lift which is a clique of $C(\Sigma)$.
\end{proposition}

\begin{proof}
A clique of $C(\Sigma)$ containing both states above one vertex cannot contain a third vertex: for a fixed state above $v$, simplicity of $\Sigma$ makes it adjacent to exactly one of the two states above any other $u$. Hence every clique of $C(\Sigma)$ of order at least three has the form $\{(v,s_v):v\in K\}$, where $K$ is a clique of $\Sigma$. Its adjacency relations give
\[
  \sigma(uv)=-s_u s_v
  \qquad(u,v\in K),
\]
which says precisely that $\Sigma[K]$ is antibalanced. Conversely, such a labeling lifts $K$ to a clique. The vertical edges account for the term $2$.
\end{proof}

The fractional colouring correspondence is exact.

\begin{proposition}[Kuffner et al.~\cite{KuffnerEtAlFractionalColoring}]
\label{prop:fractional-identity}
For every signed graph $\Sigma$,
\[
  \chi_f(C(\Sigma))=2\chifb(\Sigma).
\]
\end{proposition}

\begin{proof}
Let $(w_I)$ be a fractional colouring of $C(\Sigma)$. By Proposition~\ref{prop:balanced-independent}, the projection of every independent set is balanced. Assign to a balanced set $B$ the weight
\[
  y_B=\frac12\sum_{I:\pi(I)=B}w_I.
\]
The two states above each vertex together receive weight at least $2$, so $(y_B)$ is a fractional balanced colouring of half the total weight.

Conversely, for each balanced set $B$ choose a Harary labeling $h_B$. The two sets
\[
  I_B^+=\{(v,h_B(v)):v\in B\},
  \qquad
  I_B^-=\{(v,-h_B(v)):v\in B\}
\]
are independent. Assigning the weight of $B$ to each of them gives a fractional colouring of $C(\Sigma)$ of twice the total weight.
\end{proof}


\section{The theta parameter and the conflict graph}
\label{sec:equivalence}

We now identify the signed parameter with an ordinary theta number. The proof is the block version of the SDP symmetry encoded by the reflection.

\begin{theorem}
\label{thm:main-equivalence}
For every nonempty signed graph $\Sigma$,
\[
  \tb(\Sigma)=\frac12\tbar(C(\Sigma)).
\]
\end{theorem}

\begin{proof}
Let $(t,X,Y)$ be feasible for $\mathrm P_\Sigma$, and order the vertices of $C(\Sigma)$ by the two fibres. Define
\[
  M=
  \begin{pmatrix}
    X+Y&X-Y\\
    X-Y&X+Y
  \end{pmatrix}.
\]
The orthogonal change of basis
\[
  U=\frac1{\sqrt2}
  \begin{pmatrix}I&I\\I&-I\end{pmatrix}
\]
gives $U^{\mathsf T}MU=2X\oplus2Y$, so $M\succeq0$. Its diagonal entries are $2t-1$, its vertical entries are $-1$, and the remaining conflict edges have entry $-1$ by the signed constraints in $(\mathrm P_\Sigma)$. Thus $M$ is feasible for~\eqref{eq:tbar-primal} on $C(\Sigma)$ with value $2t$.

Conversely, average any feasible matrix $M$ for~\eqref{eq:tbar-primal} over the fibre involution, that is, if $P$ is the permutation matrix defining the fibre involution \eqref{involution}, replace $M$ by $\widehat M = (M+PMP)/2$. Note that it is still feasible for \eqref{eq:tbar-primal}. It then has the form
\[
  \widehat M=\begin{pmatrix}K&L\\L&K\end{pmatrix}.
\]
Set
\[
  X=\frac{K+L}{2},
  \qquad
  Y=\frac{K-L}{2}.
\]
Block diagonalization gives $X,Y\succeq0$. If the theta value is $\lambda$, the diagonal and vertical constraints give $X_{vv}=\lambda/2-1$ and $Y_{vv}=\lambda/2$, while the same-fibre and cross-fibre edge constraints give the two signed equations of $\mathrm P_\Sigma$. Hence $t=\lambda/2$ is feasible in $(\mathrm P_\Sigma)$.
\end{proof}

\begin{corollary}
\label{cor:monotonicity}
The value $\tb(\Sigma)$ depends only on the switching class of $\Sigma$. Moreover,
\[
  \Sigma\longrightarrow\Pi
  \quad\Longrightarrow\quad
  \tb(\Sigma)\leq\tb(\Pi).
\]
\end{corollary}

\begin{proof}
Switching gives an isomorphic conflict graph, and a signed homomorphism induces a homomorphism between conflict graphs. The result follows from Theorem~\ref{thm:main-equivalence} and homomorphism monotonicity of $\tbar$.
\end{proof}

\begin{corollary}
\label{cor:signed-sandwich}
For every nonempty signed graph $\Sigma$,
\[
  \frac12\omega(C(\Sigma))
  \leq\tb(\Sigma)
  \leq\chifb(\Sigma)
  \leq\chib(\Sigma).
\]
If $\Sigma$ is simple, the first term is $\frac12\max\{2,\omega_{\mathrm{ab}}(\Sigma)\}$.
\end{corollary}

\begin{proof}
Divide the ordinary sandwich theorem for $C(\Sigma)$ by two and use Proposition~\ref{prop:fractional-identity}. The simple signed graph statement follows from Proposition~\ref{prop:conflict-cliques}.
\end{proof}

\begin{corollary}
\label{cor:normalizations}
For nonempty signed graphs, the following statements hold.
\begin{enumerate}[(a)]
  \item $\tb(\Sigma)=1$ if and only if $\Sigma$ is balanced.
  \item $\tb(\cB_p)=p$ for every integer $p\geq1$.
  \item If $(G,\pm)$ is obtained from an ordinary graph $G$ by replacing every
  edge by a positive--negative digon, then
  \[
    \tb(G,\pm)=\tbar(G).
  \]
  \item For disjoint unions,
  \[
    \tb(\Sigma\mathbin{\dot\cup}\Pi)
    =\max\{\tb(\Sigma),\tb(\Pi)\}.
  \]
\end{enumerate}
\end{corollary}

\begin{proof}
The conflict graph is bipartite if and only if $\Sigma$ is balanced. Indeed, a Harary labeling $h$ gives the bipartition according to the sign of $s\,h(v)$ at the state $(v,s)$; conversely, the vertical edges force the two states in each fibre into opposite parts, and the resulting labels satisfy $\sigma(uv)=h(u)h(v)$. Since $C(\Sigma)$ contains a vertical edge, the strict vector formulation gives $\tbar(C(\Sigma))=2$ if and only if it is bipartite. This proves (a). For (b), $C(\cB_p)=K_{2p}$. For (c), both edge signs in $\mathrm P_\Sigma$ force $X_{uv}=-1$ and $Y_{uv}=0$ on every edge of $G$; thus the $X$-block is the ordinary program for $\tbar(G)$, while one may take $Y=tI$. Finally, conflict graphs commute with disjoint unions and $\tbar$ of a disjoint union is the maximum of the two values.
\end{proof}

\begin{example}
For the all-negative triangle $(K_3,-)$, the conflict graph is the triangular prism. Its clique number and chromatic number are both $3$, and hence
\[
  \tb(K_3,-)=\frac32.
\]
This gives a basic nonintegral example and agrees with its fractional balanced chromatic number.
\end{example}


\section{Balanced convex corners}
\label{sec:corners}

This section develops the signed analogue of
\begin{equation}
  \STAB(G)\subseteq\THbody(G)\subseteq\QSTAB(G). \label{sandwich}
\end{equation}
  
\subsection{The balanced induced subgraph polytope}

Let
\[
  \mathcal B(\Sigma)=\{B\subseteq V(\Sigma):\Sigma[B]\text{ is balanced}\}.
\]
The \emph{balanced induced subgraph polytope} was introduced by Barahona and Mahjoub~\cite{BarahonaMahjoub1989,BarahonaMahjoub1994}. We use the notation
\[
  \BSTAB(\Sigma)
  :=\conv\{\one_B:B\in\mathcal B(\Sigma)\}.
\]
Its support function is the maximum-weight balanced induced subgraph problem.

Define the linear map
\[
  p:\R^{V(\Sigma)\times\{\pm1\}}\longrightarrow\R^{V(\Sigma)},
  \qquad
  p(y)_v=y_{(v,+)}+y_{(v,-)}.
\]

\begin{lemma}
\label{lem:BSTAB-projection}
For every signed graph $\Sigma$,
\[
  \BSTAB(\Sigma)=p\bigl(\STAB(C(\Sigma))\bigr).
\]
\end{lemma}

\begin{proof}
By Proposition~\ref{prop:balanced-independent}, the projection of the incidence vector of an independent set is the incidence vector of a balanced set. Conversely, every balanced set has a Harary labeling and hence lifts to an independent set. Taking convex hulls proves the identity.
\end{proof}

\subsection{The signed theta and clique bodies}

We define
\begin{equation}
  \BTH(\Sigma):=p\bigl(\THbody(C(\Sigma))\bigr),
  \qquad
  \BQSTAB(\Sigma):=p\bigl(\QSTAB(C(\Sigma))\bigr).
\label{eq:signed-bodies}
\end{equation}
These are convex corners in $\R^{V(\Sigma)}$. For completeness, if $x=p(y)$ and $0\leq x'\leq x$, scale the two coordinates above each $v$ by $x'_v/x_v$ (with the evident convention when $x_v=0$). Lower-comprehensiveness of the lifted corner then shows that $x'$ belongs to its image. The three bodies satisfy
\begin{equation}
  \BSTAB(\Sigma)
  \subseteq\BTH(\Sigma)
  \subseteq\BQSTAB(\Sigma).
\label{eq:signed-body-sandwich}
\end{equation}
The first inclusion also follows from the direct formulation below.

\begin{theorem}
\label{thm:BTH-intrinsic}
A vector $x\in\R^{V(\Sigma)}$ belongs to $\BTH(\Sigma)$ if and only if there are matrices $M,N\in\mathbb S^{V(\Sigma)}$ such that
\begin{equation}
  \begin{pmatrix}1&x^{\mathsf T}\\x&M\end{pmatrix}\succeq0,
  \qquad
  N\succeq0,
\label{eq:BTH-psd}
\end{equation}
\begin{equation}
  M_{vv}=N_{vv}=x_v
  \qquad(v\in V(\Sigma)),
\label{eq:BTH-diag}
\end{equation}
and, for every signed edge $e=uv$,
\begin{equation}
  N_{uv}=\sigma(e)M_{uv}.
\label{eq:BTH-edge}
\end{equation}
If both signs occur between $u$ and $v$, then $M_{uv}=N_{uv}=0$.
\end{theorem}

\begin{proof}
Let $y\in\THbody(C(\Sigma))$ project to $x$, and average a theta-body certificate under the fibre involution. We may assume
\[
  y=\frac12(x,x),
  \qquad
  Z=\begin{pmatrix}A&B\\B&A\end{pmatrix}.
\]
The vertical edges give $B_{vv}=0$, while $A_{vv}=x_v/2$. In the even and odd fibre coordinates, positive semidefiniteness is equivalent to
\[
  \begin{pmatrix}1&x^{\mathsf T}\\x&2(A+B)\end{pmatrix}\succeq0,
  \qquad
  2(A-B)\succeq0.
\]
Set $M=2(A+B)$ and $N=2(A-B)$. A positive edge makes the corresponding cross-fibre entry $B_{uv}$ vanish, and hence $N_{uv}=M_{uv}$; a negative edge makes $A_{uv}$ vanish, and hence $N_{uv}=-M_{uv}$. This gives \eqref{eq:BTH-psd}--\eqref{eq:BTH-edge}. Reversing the construction gives a symmetric theta-body lift, proving the converse.
\end{proof}

For simple signed graphs, the clique body has an intrinsic description. A set $K\subseteq V(\Sigma)$ is an \emph{antibalanced clique} when the underlying induced graph is complete and $\Sigma[K]$ is antibalanced.

\begin{proposition}
\label{prop:BQSTAB-intrinsic}
If $\Sigma$ is a simple signed graph, then
\[
  \BQSTAB(\Sigma)
  =\left\{x\in[0,1]^{V(\Sigma)}:
  x(K)\leq2\text{ for every antibalanced clique }K\right\}.
\]
\end{proposition}

\begin{proof}
A point of the projected clique relaxation may be symmetrized, so it has the lift $y_{(v,+)}=y_{(v,-)}=x_v/2$. If $K$ is antibalanced, choose signs $s_v$ such that $\sigma(uv)=-s_u s_v$. Then $\{(v,s_v):v\in K\}$ is a clique of $C(\Sigma)$, giving $x(K)\leq2$. The vertical clique gives $x_v\leq1$.

Conversely, consider a clique $Q$ of $C(\Sigma)$. If it contains both states above one vertex, simplicity implies that $Q$ is exactly that vertical edge. Otherwise its projection is injective and, by Proposition~\ref{prop:conflict-cliques}, is an antibalanced clique. The stated inequalities therefore imply $y(Q)\leq1$ for every clique $Q$.
\end{proof}

\subsection{The scalar gauge and recovery of the ordinary hierarchy}

The signed theta parameter is the all-ones gauge of the body just defined. This is the body-level counterpart of Theorem~\ref{thm:main-equivalence}.

\begin{proposition}
\label{prop:gauge-BTH}
For every nonempty signed graph $\Sigma$,
\[
  \tb(\Sigma)=\gamma_{\BTH(\Sigma)}(\one).
\]
\end{proposition}

\begin{proof}
The condition $\one/\lambda\in\BTH(\Sigma)$ is equivalent, after averaging under the fibre involution, to
\[
  \frac{1}{2\lambda}\one\in\THbody(C(\Sigma)).
\]
Thus $2\lambda$ is feasible for the gauge of $\one$ in $\THbody(C(\Sigma))$. By~\eqref{eq:gauge-theta}, the minimum such value is $\tbar(C(\Sigma))$. The conclusion follows from Theorem~\ref{thm:main-equivalence}.
\end{proof}

The construction recovers the ordinary stable-set hierarchy.

\begin{theorem}
\label{thm:body-recovery}
Let $(G,\pm)$ be obtained from a simple graph $G$ by replacing every edge by a positive--negative digon. Then
\[
  \BSTAB(G,\pm)=\STAB(G),
  \qquad
  \BTH(G,\pm)=\THbody(G),
  \qquad
  \BQSTAB(G,\pm)=\QSTAB(G).
\]
\end{theorem}

\begin{proof}
A vertex set is balanced in $(G,\pm)$ if and only if it is stable in $G$, so the first equality is immediate. In Theorem~\ref{thm:BTH-intrinsic}, both sign constraints on an edge force $M_{uv}=N_{uv}=0$. Hence every point of $\BTH(G,\pm)$ has an ordinary theta-body certificate $M$. Conversely, from an ordinary certificate $M$ one obtains a signed certificate by taking $N=\Diag(x)$. This proves the second equality.

Finally, $C(G,\pm)$ is obtained by replacing every vertex of $G$ by a vertical $K_2$ and joining two fibres completely whenever the corresponding vertices are adjacent. Above every clique $K$ of $G$, all $2|K|$ states form a clique, so the projected inequalities include $x(K)\leq1$. Conversely, the symmetric lift $y_{(v,\pm)}=x_v/2$ satisfies every clique inequality whenever $x\in\QSTAB(G)$.
\end{proof}

\subsection{The all-negative specialization}
\label{sec:negative}

Let $-G$ denote the signed graph obtained by assigning sign $-1$ to every edge of a simple graph $G$, together with the usual positive loops. A signed subgraph of $-G$ is balanced precisely when its underlying graph is bipartite. Thus this specialization corresponds to the maximum $2$-colourable induced subgraph problem. Semidefinite relaxations of the maximum $k$-colourable induced subgraph problem, including the generalized theta number $\vartheta_k$, are studied in~\cite{KuryatnikovaSotirovVera,SinjorgoSotirov}; recent strengthened relaxations appear in~\cite{BarkelSotirov}.

Let
\[
  \BIP(G)=\conv\{\one_U:G[U]\text{ is bipartite}\},
\]
and let $\alpha_2(G)$ be the maximum cardinality of such a set $U$.

\begin{proposition}
\label{prop:all-negative}
For every simple graph $G$,
\begin{enumerate}[(a)]
  \item $\BSTAB(-G)=\BIP(G)$;
  \item $C(-G)=K_2\cart G$;
  \item
  \[
    \max\{\one^{\mathsf T}x:x\in\BTH(-G)\}
    =\vartheta(K_2\cart G);
  \]
  \item
  \[
    \BQSTAB(-G)
    =\{x\in[0,1]^{V(G)}:x(K)\leq2
      \text{ for every clique }K\text{ of }G\}.
  \]
\end{enumerate}
\end{proposition}

\begin{proof}
Part (a) is the balance characterization above. Negative edges produce same-fibre edges in the conflict graph, while the positive loops produce the vertical matching, which is exactly the Cartesian product in (b). For an ordinary graph $H$, maximizing $\one^{\mathsf T}x$ over $\THbody(H)$ gives $\vartheta(H)$. Since the objective is preserved by the projection $p$, (c) follows from the definition of $\BTH$. Finally, the antibalanced cliques of $-G$ are precisely the cliques of $G$, so (d) follows from Proposition~\ref{prop:BQSTAB-intrinsic}.
\end{proof}


\section{Balanced perfectness}
\label{sec:perfect}

For ordinary graphs, perfectness is equivalent to equality of the stable-set and clique relaxations, and also to polyhedrality of the theta body; see~\cite{GLS}. One of the key applications of the notion of perfectness is that for these graphs, both polytopes $\STAB(G)$ and $\QSTAB(G)$ can be dealt with efficiently, because the sandwich \eqref{sandwich} becomes an equality. This motivates the following signed analogue.

\begin{definition}
A signed graph $\Sigma$ is \emph{balanced-perfect} if
\[
  \BSTAB(\Sigma)=\BQSTAB(\Sigma)
\]
\end{definition}

Perfectness of the conflict graph is sufficient, but not necessary.

\begin{theorem}
\label{thm:perfectness-transfer}
If $C(\Sigma)$ is perfect, then $\Sigma$ is balanced-perfect. The converse fails in general.
\end{theorem}

\begin{proof}
For every $U\subseteq V(\Sigma)$,
\[
  C(\Sigma[U])=C(\Sigma)[U\times\{\pm1\}].
\]
If $C(\Sigma)$ is perfect, then
\[
  \STAB(C(\Sigma[U]))
  =\THbody(C(\Sigma[U]))
  =\QSTAB(C(\Sigma[U])).
\]
Applying $p$ proves that $\Sigma$ is balanced-perfect.

For the failed converse, let $\Sigma$ have vertex set $\{a,b,c,d\}$ and underlying graph $K_4-cd$. Give $bc$ negative sign and every other edge positive sign. The triangle $abc$ is negative, the triangle $abd$ is positive, and every unbalanced induced vertex set contains $\{a,b,c\}$. Therefore
\[
  \BSTAB(\Sigma)
  =\{x\in[0,1]^4:x_a+x_b+x_c\leq2\}.
\]
The only antibalanced clique of order at least three is $\{a,b,c\}$, so Proposition~\ref{prop:BQSTAB-intrinsic} gives the same description for $\BQSTAB(\Sigma)$. Thus $\Sigma$ is balanced-perfect and
\[
  \BSTAB(\Sigma)=\BTH(\Sigma)=\BQSTAB(\Sigma).
\]

Nevertheless, $C(\Sigma)$ contains the induced cycle
\[
  (a,+),(c,-),(c,+),(b,+),(d,-),(a,+).
\]
The consecutive edges arise, respectively, from the positive edge $ac$, the vertical edge at $c$, the negative edge $bc$, and the positive edges $bd$ and $ad$. There are no chords: the same-fibre pairs above $ac$ and $ab$ are nonedges, the cross-fibre pair above $bc$ is a nonedge, and $cd$ is absent. Hence $C(\Sigma)$ contains an odd hole and is not perfect.
\end{proof}

The example isolates the information lost by the projection $p$. Signed induced subgraphs correspond to induced subgraphs of $C(\Sigma)$ containing whole fibres, whereas an odd hole in the conflict graph may use both states of some vertices and only one state of others. The projected body can therefore be exact even when the full theta body is not. Characterizing this phenomenon is a genuinely signed perfectness problem.


\section{Concluding remarks}

We introduced a balanced Lov\'asz theta parameter through homomorphisms to a signed sphere with an orthogonal involution. The two eigenspaces of the involution lead directly to a symmetric two-matrix primal--dual SDP pair. At the scalar level, the construction is exactly one half of the chromatic-side theta number of the conflict graph, generalizing the framework of several other signed parameters. 

The more substantial object is the projected theta body $\BTH(\Sigma)$. It lies in the original signed vertex space, relaxes the balanced induced subgraph polytope, has an intrinsic two-Gram description, and has $\tb(\Sigma)$ as its all-ones gauge. Together with $\BSTAB(\Sigma)$ and $\BQSTAB(\Sigma)$, it recovers the ordinary stable-set hierarchy on signed digon graphs. The all-negative specialization connects it to maximum induced bipartite subgraphs and to the Cartesian-product theta bound. The balanced-perfect example shows that projection can erase imperfectness of the conflict graph, so the signed body theory has structural questions not captured by ordinary perfectness alone.

We close with two concrete problems.

\begin{problem}
Characterize balanced-perfect signed graphs. In particular, find minimal signed obstructions to the equality $\BSTAB=\BQSTAB$, and determine natural classes on which $\BSTAB=\BTH$ is exact.
\end{problem}

\begin{problem}
Determine the scalar parameter and the full signed theta body for the signed Kneser and signed Schrijver graphs of~\cite{KuffnerEtAl}. A useful refinement may need to retain the ranks of the fixed and anti-fixed Gram blocks, since the scalar optimum forgets the projective information used in the signed Borsuk--Ulam arguments.
\end{problem}


\vspace{1cm}

\subsection*{Statement of AI use}

ChatGPT was used to help organize the literature, test algebraic formulations, and review the final version of this text. All mathematical claims, priority assessments, and final editorial decisions remain the responsibility of the author.

\subsection*{Funding}

The author acknowledges support from Fundação de Amparo à Pesquisa do Estado de Minas Gerais (FAPEMIG) and from Conselho Nacional de Desenvolvimento Científico e Tecnológico (CNPq).

\bibliographystyle{plain}
\bibliography{signedlovaszthetav2}

\end{document}